\documentclass[a4paper,11pt]{article}
\usepackage{amsmath,indentfirst,amsfonts,dsfont,amssymb,amsthm,graphicx,enumerate,mathrsfs,color,bm,url}
\usepackage{algorithm}               
\usepackage{algorithmic}             

\numberwithin{equation}{section}

\newtheorem{definition}{Definition}[section]
\newtheorem{theorem}{Theorem}[section]
\newtheorem{lemma}{Lemma}[section]
\newtheorem{proposition}{Proposition}[section]
\newtheorem{corollary}{Corollary}[section]

\newtheorem{example}{Example}[section]

\title{\bf Stochastic $R_0$ Tensors to Stochastic Tensor Complementarity Problems}
\author{
Maolin Che\thanks{E-mail: chncml@outlook.com and cheml@swufe.edu.cn. School of
Economic Mathematics, Southwest University of Finance and Economics, Chengdu, 611130, P. R. of
China.}
\and
Liqun Qi\thanks{E-mail: liqun.qi@polyu.edu.hk.
 Department of Applied Mathematics, the Hong Kong Polytechnic University, Hong Kong.
E-mail: liqun.qi@polyu.edu.hk. This author is supported by the Hong Kong Research Grant Council (Grant No.
PolyU 501913, 15302114, 15300715 and 15301716).}\and Yimin Wei\thanks{
E-mail: ymwei@fudan.edu.cn and yimin.wei@gmail.com. School of
Mathematical Sciences and Shanghai Key Laboratory of Contemporary
Applied Mathematics, Fudan University, Shanghai, 200433, P. R. of
China. This author is supported by the National Natural Science
Foundation of China under grant 11771099 and   Shanghai Municipal Education Commission.}
}
\date{}

\begin{document}
\maketitle
\begin{abstract}
The main purpose of this paper is devoted to an introduction of the stochastic tensor complementarity problem. We consider the expected residual minimization formulation of the stochastic tensor complementarity problem. We show that the solution set of the expected residual minimization problem is nonempty and bounded, if the associated tensor is an $R_0$ tensor. We also prove that the associated tensor being a stochastic $R_0$ tensor is a necessary and sufficient condition for the solution set of the expected residual minimization problem to be nonempty and bounded.
  \bigskip


  \bigskip

\end{abstract}
\section{Introduction}
Over the past decade, the research of finite-dimensional variational inequalities and complementarity problems \cite{Jacobians_ncp,ncpv1,NCP,NCPalgorithm} has been rapidly developed in the theory of existence, uniqueness and sensitivity of solutions, theory of algorithms, and the application of these techniques to transportation planning, regional science, socio-economic analysis, energy modeling and game theory. In many practical applications, complementarity problems
often involve uncertain data. However, references on stochastic complementarity
problems \cite{CWZ12,SNCP_introduction1,SNCP_introduction2,SNCP_introduction3,ZC14} are relatively scarce, compared with stochastic optimization problems for which abundant results are available in the literature; see \cite{SNCP_introduction4,SNCP_introduction5} in particular for simulation-based approaches in stochastic optimization.

The tensor complementarity problem (TCP), which is a natural generalization of the linear complementarity problem (LCP) and a special case of the nonlinear complementarity problem (NCP), is a new topic emerged from the tensor community, inspired by the growing research on structured tensors. The TCP is widely used in nonlinear compressed sensing, commutations, DNA microarrays, multi-person game and so on. The readers can be recommended \cite{tcp1,tcp2,DING2018336,shouqiangtensor,tcp7,tcp3,liu2017tensor,tcp4,tcp5,tcp6,wang2016exceptionally} for a thorough survey of the existence of the solution of the TCP problem. The interested readers can refer to \cite{shouqiangtensor,han2018continuation,tcp3,liu2017tensor,wang2018gradient,xie2017iterative} for numerical algorithms for tensor complementarity problems.

Song and Qi \cite{tcp5} introduced the definition of $R_0$ tensors and considered the solvability of the tensor complementarity problem with $R_0$ tensors. The linear complementarity problem with stochastic $R_0$ matrices are investigated in \cite{slcp_fang_SJO_2007}. In this paper, we study the stochastic tensor complementarity problem (STCP) with any subset $\Omega\subset\mathbb{R}^I$, which can be viewed as the generalization of the stochastic linear complementarity problem to the tensor case. If $\Omega$ only contains a single realization, then the STCP reduced to the standard TCP. In many potential applications, some data in the TCP cannot be known with certainty. The STCP is aimed at a practical treatment of the TCP under uncertainty. However, only a little attention has been paid to the STCP in the literature. We make the following contributions:
\begin{enumerate}[(1)]
\item We  present the formula for the stochastic tensor complementarity problem (STCP) and study the expected residual minimization formulations for the STCP, which employs an NCP function. In order to study the STCP, we introduce the definition of stochastic $R_0$ tensors.
\item We show that a sufficient condition for the existence of a solution to the expected residual minimization problem and its discrete approximations is that there is an observation $\bm{\omega}$ such that the coefficient tensor $\mathcal{A}(\bm{\omega})$ is an $R_0$ tensor.
\item We prove that the involved tensor being a stochastic $R_0$ tensor is a necessary and sufficient condition for the solution set of the expected residual minimization problem to be nonempty and bounded.
\end{enumerate}

Throughout this paper, the norm $\|\cdot\|_2$ denotes the Euclidean norm and we use $\mathbb{R}^I_+=\{\mathbf{x}\in\mathbb{R}^I:x_i\geq0,i=1,2,\dots,I\}$ to denote the nonnegative orthant. For a given vector $\mathbf{x}\in\mathbb{R}^I$, we denote $\mathbb{I}(\mathbf{x})=\{i:x_i=0,i=1,2,\dots,I\}$ and $\mathbb{J}(\mathbf{x})=\{i:x_i\neq0,i=1,2,\dots,I\}$.
For given two vectors $\mathbf{x},\mathbf{y}\in\mathbb{R}^I$, $\min(\mathbf{x},\mathbf{y})$ denotes the vector with components $\min(x_i,y_i)$ for all $i=1,2,\dots,I$. We use $\mathbf{0}_I$ and $\mathbf{I}_I$ to denote the zero vector in $\mathbb{R}^I$ and the identity matrix in $\mathbb{R}^{I\times I}$, respectively.

The following notations will be adopted. We assume that $I$, $J$, and $N$ will be reserved to denote the index upper bounds, unless stated otherwise. We use small letters $x,u,v,\dots$ for scalars, small bold letters $\mathbf{x},\mathbf{u}, \mathbf{v},\dots$ for vectors, bold capital letters $\mathbf{A},\mathbf{B},\mathbf{C},\dots$ for matrices, and calligraphic letters $\mathcal{A},\mathcal{B},\mathcal{C},\dots$ for higher-order tensors. This notation is consistently used for lower-order parts of a given structure. For example, the entry with row index $i$ and column index $j$ in a matrix ${\bf A}$, i.e., $({\bf A})_{ij}$, is symbolized by $a_{ij}$ (also $(\mathbf{x})_i=x_i$ and $(\mathcal{A})_{i_1i_2 \dots i_N}=a_{i_1i_2\dots i_N}$).

The remainder of this paper is organized as follows. In Section \ref{STCP:section2}, we introduce the stochastic tensor complementarity problem and formulate its expected residual minimization. In Section \ref{STCP:section3}, we introduce the definition of $R_0$ tensors and present a sufficient condition for the nonemptieness of the solution set of the expected residual minimization problem associated with a given stochastic tensor complementarity problem. In Section \ref{STCP:section4}, we provide the definition of a stochastic  $R_0$ tensor and derive a sufficient and necessary condition for the nonemptieness of the solution set of this expected residual minimization problem. We conclude our paper with some remarks about the STCP in Section \ref{STCP:section5}.
\section{Stochastic tensor complementarity problems}\label{STCP:section2}
The set of all $N$th order $I$-dimensional real tensors is denoted by $T_{N,I}$. We first introduce two denotations \cite{Qiliqun} as follows. For any $\mathcal{A}\in T_{N,I}$ and $\mathbf{x}\in\mathbb{R}^I$, $\mathcal{A}\mathbf{x}^{N-1}$ is an $I$-dimensional real vector whose $i$th component is
$(\mathcal{A}\mathbf{x}^{N-1})_i=\sum_{i_2,\dots,i_N=1}^Ia_{ii_2\dots i_N}x_{i_2}\dots x_{i_N}$,
and $\mathcal{A}\mathbf{x}^N$ is a scalar given by
$\mathcal{A}\mathbf{x}^{N}=\sum_{i_1,i_2,\dots,i_N=1}^Ia_{i_1i_2\dots i_N}x_{i_1}x_{i_2}\dots x_{i_N}$.
The Frobenius norm of $\mathcal{A}$ is given by $\|\mathcal{A}\|_{F}=\sqrt{\langle\mathcal{A},\mathcal{A}\rangle}$ and the scalar product $\langle\mathcal{A},\mathcal{B}\rangle$ is defined by  $\langle\mathcal{A},\mathcal{B}\rangle=\sum_{i_{1},i_2,\dots,i_{N}=1}^{I}b_{i_{1}i_{2}\dots i_{N}}a_{i_{1}i_{2}\dots i_{N}}$ \cite{nonnegative,Kolda}.

For a given $\mathcal{A}\in T_{N,I}$ and $\mathbf{q}\in\mathbb{R}^I$, the tensor complementarity problem, denoted by ${\rm TCP}(\mathcal{A},\mathbf{q})$, is to find $\mathbf{x}\in\mathbb{R}^I_+$ such that
$$\mathcal{A}\mathbf{x}^{N-1}+\mathbf{q}\in\mathbb{R}^I_+,\quad \mathcal{A}\mathbf{x}^{N}+\mathbf{x}^\top\mathbf{q}=0,$$
or to show that no such vector exists. Let $(\Omega,\mathfrak{F},P)$ be a probability space, where $\Omega$ is a subset of $\mathbb{R}^I$, and $\mathfrak{F}$ is a $\sigma$-algebra generated by $\{\Omega\cap\mathbb{U}:\mathbb{U}\text{ is an open set in }\mathbb{R}^I\}$. For given $\mathcal{A}(\bm{\omega})\in T_{N,I}$ and $\mathbf{q}(\bm{\omega})\in\mathbb{R}^{I}$ with $\bm{\omega}\in\Omega$, the stochastic tensor complementarity problem (STCP), denoted by ${\rm STCP}(\mathcal{A}(\bm{\omega}),\mathbf{q}(\bm{\omega}))$, is to find $\mathbf{x}\in\mathbb{R}^I_+$ such that
\begin{equation*}
\mathcal{A}(\bm{\omega})\mathbf{x}^{N-1}+\mathbf{q}(\bm{\omega})\in\mathbb{R}^I_+,\quad \mathcal{A}(\bm{\omega})\mathbf{x}^N+\mathbf{x}^\top\mathbf{q}(\bm{\omega})=0,
\end{equation*}
or to show that no such vector exists.

When $N=2$, the STCP reduces to the stochastic linear problem (SLCP):
\begin{equation*}
\mathbf{x}\in\mathbb{R}^I_+,\quad \mathbf{A}(\bm{\omega})\mathbf{x}+\mathbf{q}(\bm{\omega})\in\mathbb{R}^I_+,\quad \mathbf{x}^\top\mathbf{A}(\bm{\omega})\mathbf{x}+\mathbf{x}^\top\mathbf{q}(\bm{\omega})=0,
\end{equation*}
where $\mathbf{A}(\bm{\omega})\in \mathbb{R}^{I\times I}$ and $\mathbf{q}(\bm{\omega})\in\mathbb{R}^{I}$ for $\bm{\omega}\in\Omega$. The STCP is a special case of the stochastic nonlinear complementarity problem (SNCP) \cite{CZF09,LCF07,LCF09,ZC08,ZC09}
\begin{equation*}
\mathbf{x}\in\mathbb{R}^I_+,\quad F(\mathbf{x},\bm{\omega})\in\mathbb{R}^I_+,\quad \mathbf{x}^\top F(\mathbf{x},\bm{\omega})=0,
\end{equation*}
where $F(\mathbf{x},\bm{\omega}):\mathbb{R}^I\times \Omega\rightarrow\mathbb{R}^I$ and $\mathbf{q}(\bm{\omega})\in\mathbb{R}^{I}$ for $\bm{\omega}\in\Omega$.

Throughout this paper, we assume that $\mathcal{A}(\bm{\omega})$ and $\mathbf{q}(\bm{\omega})$ are measured functions of $\bm{\omega}$ with the following property:
\begin{equation*}
\mathbb{E}\{\|\mathcal{A}(\bm{\omega})\|_F\}<\infty,\quad
\mathbb{E}\{\|\mathbf{q}(\bm{\omega})\|_2\}<\infty,
\end{equation*}
where $\mathbb{E}\{\cdot\}$ stands for the expectation with respect to $\bm{\omega}$.

Similar to the case of general nonlinear complementarity problems, there does not exist $\mathbf{x}\in\mathbb{R}^I_+$ satisfying the ${\rm STCP}(\mathcal{A}(\bm{\omega}),\mathbf{q}(\bm{\omega}))$ for almost all $\bm{\omega}\in\Omega$. A deterministic formulation for the STCP provides a decision vector which is optimal in a certain sense. Different deterministic formulations may yield different solutions that are optimal in different senses.

G\"{u}rkan, \"{O}zge, and Robinson \cite{SNCP_introduction2} considered the sample-path approach for stochastic variational inequalities and provided convergence theory and applications for the approach. When it applied to the ${\rm STCP}(\mathcal{A}(\bm{\omega}),\mathbf{q}(\bm{\omega}))$, the approach is the same as the {\it expected value} (EV) method, which uses the expected function of the random function $\mathcal{A}(\bm{\omega})\mathbf{x}^{N-1}+\mathbf{q}(\bm{\omega})$ and solves the deterministic problem
\begin{equation*}
\mathbf{x}\in\mathbb{R}^I_+,\quad \mathbb{E}\{\mathcal{A}(\bm{\omega})\mathbf{x}^{N-1}+\mathbf{q}(\bm{\omega})\}
\in\mathbb{R}^I_+,\quad \mathbb{E}\{\mathcal{A}(\bm{\omega})\mathbf{x}^N
+\mathbf{x}^\top\mathbf{q}(\bm{\omega})\}=0.
\end{equation*}

 Chen and Fukushima \cite{slcp_chen_MOR_2005} proposed a deterministic formulation called the {\it expected residual minimization} (ERM) method, which is to find a vector $\mathbf{x}\in\mathbb{R}^I$ that minimizes the expected residual of the stochastic linear complementarity problem. Similarly, we can also propose a deterministic formulation called the {\it expected residual minimization} (ERM) method, which is to find a vector $\mathbf{x}\in\mathbb{R}^I$ that minimizes the expected residual of the ${\rm STCP}(\mathcal{A}(\bm{\omega}),\mathbf{q}(\bm{\omega}))$, i.e.,
\begin{equation}\label{STCP:equation1}
\min_{\mathbf{x}\in\mathbb{R}^I_+}
\mathbb{E}\{\|\Phi(\mathbf{x},\bm{\omega})\|_2^2\},
\end{equation}
where $\Phi:\mathbb{R}^I\times \Omega\rightarrow\mathbb{R}^{I}$ is defined by
\begin{equation*}
\Phi(\mathbf{x},\bm{\omega})=\begin{pmatrix}
\phi((\mathcal{A}(\bm{\omega})\mathbf{x}^{N-1})_1+(\mathbf{q}(\bm{\omega}))_1,x_1)\\
\phi((\mathcal{A}(\bm{\omega})\mathbf{x}^{N-1})_2+(\mathbf{q}(\bm{\omega}))_2,x_2)\\
\vdots\\
\phi((\mathcal{A}(\bm{\omega})\mathbf{x}^{N-1})_I+(\mathbf{q}(\bm{\omega}))_I,x_I)
\end{pmatrix}.
\end{equation*}
Here $\phi:\mathbb{R}^2\rightarrow\mathbb{R}$ is an NCP function which has the property,
\begin{equation*}
\phi(a,b)=0\Leftrightarrow a\geq0,\quad b\geq0,\quad ab=0.
\end{equation*}

For the case of $N=2$, if $\Omega$ has only one realization, then the ERM problem (\ref{STCP:equation1}) associated with an STCP reduces to the standard LCP and the solvability of (\ref{STCP:equation1}) does not depend on the choice of NCP functions. Example 1 in \cite{slcp_chen_MOR_2005} can show that the solution set of ${\rm STCP}(\mathcal{A}(\cdot),\mathbf{q}(\cdot))$ being nonempty and bounded for some $\overline{\bm{\omega}}\in\Omega$ does not imply that the ${\rm ERM}(\mathcal{A}(\cdot),\mathbf{q}(\cdot))$ has a solution.

Various NCP functions have been studied for solving complementarity problems \cite{ncpv1}. In this paper, we will concentrate on the ``min'' function
$\phi(a,b)=\min(a,b)$,
and the Fischer-Burmeister (FB) function \cite{fb_function},
$\phi(a,b)=a+b-\sqrt{a^2+b^2}.$

All NCP functions including the ``$\min$" function and FB function are equivalent in the sense that they can reformulate any complementarity problem as a system of nonlinear equations having the same solution set. Moreover, some NCP functions have the same growth rate. In particular, Tseng \cite{meritfun_tseng_jota_1996} showed that the ``min'' function and the FB function satisfy
\begin{equation}\label{STCP:equation2}
\frac{2}{\sqrt{2}+2}|\min(a,b)|\leq|a+b-\sqrt{a^2+b^2}|\leq(\sqrt{2}+2)|\min(a,b)|
\end{equation}
where $a,b\in\mathbb{R}$. We use $\Phi_1(\mathbf{x},\bm{\omega})$ and $\Phi_2(\mathbf{x},\bm{\omega})$ to distinguish the function $\Phi(\mathbf{x},\bm{\omega})$ defined by the ``min'' function and the FB function, respectively.

Let ${\rm ERM}(\mathcal{A}(\cdot),\mathbf{q}(\cdot))$ denote (\ref{STCP:equation1}) and define
\begin{equation}\label{STCP:equation4}
G(\mathbf{x})=\int_{\Omega}\|\Phi(\mathbf{x},\bm{\omega})\|_2^2dF(\bm{\omega})
=\int_{\Omega}\|\Phi(\mathbf{x},\bm{\omega})\|_2^2f(\bm{\omega})d\bm{\omega},
\end{equation}
where $F(\bm{\omega})$ is the distribution function of $\bm{\omega}$ and $f(\bm{\omega})$ is its continuous probability density function. Both $F(\bm{\omega})$ and $f(\bm{\omega})$ satisfy
$\int_{\Omega}dF(\bm{\omega})=\int_{\Omega}f(\bm{\omega})d\bm{\omega}=1$.
We call $\mathbf{x}_*\in\mathbb{R}^I_+$ a {\it local solution} of the ${\rm ERM}(\mathcal{A}(\cdot),\mathbf{q}(\cdot))$, if there is $\gamma>0$ such that $G(\mathbf{x})\geq G(\mathbf{x}_*)$ for all $\mathbf{x}\in\mathbb{R}^I_+\cap\mathbb{B}(\mathbf{x}_*,\gamma)$, where $\mathbb{B}(\mathbf{x}_*,\gamma)
=\{\mathbf{x}\in\mathbb{R}^I:\|\mathbf{x}-\mathbf{x}_*\|_2\leq\gamma\}$, and call $\mathbf{x}_*$ a {\it global solution} of the ${\rm ERM}(\mathcal{A}(\cdot),\mathbf{q}(\cdot))$, if $G(\mathbf{x})\geq G(\mathbf{x}_*)$ for all $\mathbf{x}\in\mathbb{R}^I_+$. Then ${\rm ERM}(\mathcal{A}(\cdot),\mathbf{q}(\cdot))$ can be rewritten as
\begin{equation*}
\min\quad G(\mathbf{x})\quad\quad\text{subject to}\quad \mathbf{x}\in\mathbb{R}^I_+.
\end{equation*}
Denote two sets of function $G(\mathbf{x})$ by
\begin{equation*}
\mathbb{D}(\gamma):=\{\mathbf{x}\in\mathbb{R}^I_+:G(\mathbf{x})\leq\gamma\},\quad
\mathbb{D}^{C}(\gamma):=\{\mathbf{x}\in\mathbb{R}^I_+:G(\mathbf{x})\geq\gamma\}.
\end{equation*}
\section{$R_0$ tensors for STCPs}\label{STCP:section3}

\subsection{A property of $R_0$ tensors}
Recall that $\mathbf{A}\in \mathbb{R}^{I\times I}$ is called an $R_0$ matrix  \cite[Definition 3.9.20]{lcp}, if
\begin{equation*}
\mathbf{x}\in\mathbb{R}^I_+,\quad \mathbf{A}\mathbf{x}\in\mathbb{R}^I_+,\quad
\mathbf{x}^\top\mathbf{A}\mathbf{x}=0\quad
\Rightarrow\quad \mathbf{x}=\mathbf{0}_I.
\end{equation*}
Similarly, an $R_0$ tensor is formally defined as follows, which is quoted from \cite[Definition 2.1]{song2016properties}.
\begin{definition}
$\mathcal{A}\in T_{N,I}$ is called an $R_0$ tensor, if
\begin{equation*}
\mathbf{x}\in\mathbb{R}^I_+,\quad \mathcal{A}\mathbf{x}^{N-1}\in\mathbb{R}^I_+,\quad
\mathcal{A}\mathbf{x}^{N}=0\quad
\Rightarrow\quad \mathbf{x}=\mathbf{0}_I.
\end{equation*}
\end{definition}
The following lemma will be used to study the solution set of ${\rm ERM}(\mathcal{A}(\cdot),\mathbf{q}(\cdot))$.
\begin{lemma}\label{STCP:lemma1}
If $\mathcal{A}(\overline{\bm{\omega}})\in T_{N,I}$ is an $R_0$ tensor for some $\overline{\bm{\omega}}\in\Omega$, then there is a closed sphere $\mathbb{B}(\overline{\bm{\omega}},\delta)=\{\bm{\omega}\in\Omega:
\|\bm{\omega}-\overline{\bm{\omega}}\|_2\leq\delta\}$ with $\delta>0$ such that for every $\bm{\omega}\in\overline{\mathbb{B}}:=\mathbb{B}(\overline{\bm{\omega}},\delta)\cap\Omega$, $\mathcal{A}(\bm{\omega})$ is an $R_0$ tensor.
\end{lemma}
\begin{proof}
 Suppose that this lemma is not true. Then there is a sequence $\{\bm{\omega}_k\}\subset\overline{\mathbb{B}}$ such that
\begin{equation*}
\lim_{k\rightarrow\infty}\bm{\omega}_k=\overline{\bm{\omega}}
\end{equation*}
and, for every $\mathcal{A}(\bm{\omega}_k)$, we can find nonzero $\mathbf{x}_k\in\mathbb{R}^I$ satisfying
\begin{equation*}
\mathbf{x}_k\in\mathbb{R}^I_+,\quad \mathcal{A}(\bm{\omega}_k)\mathbf{x}_k^{N-1}\in\mathbb{R}^I_+,\quad
\mathcal{A}(\bm{\omega}_k)\mathbf{x}_k^{N}=0.
\end{equation*}
Let $\mathbf{v}_k=\frac{\mathbf{x}_k}{\|\mathbf{x}_k\|_2}$. Then we have
\begin{equation*}
\mathbf{v}_k\in\mathbb{R}^I_+,\quad \|\mathbf{v}_k\|_2=1,\quad \mathcal{A}(\bm{\omega}_k)\mathbf{v}_k^{N-1}\in\mathbb{R}^I_+,\quad
\mathcal{A}(\bm{\omega}_k)\mathbf{v}_k^{N}=0.
\end{equation*}
If $k\rightarrow\infty$, then we obtain a vector $\overline{\mathbf{v}}\in\mathbb{R}^I$ such that
\begin{equation*}
\overline{\mathbf{v}}\in\mathbb{R}^I_+,\quad \|\overline{\mathbf{v}}\|_2=1,\quad \mathcal{A}(\overline{\bm{\omega}})\overline{\mathbf{v}}^{N-1}\in\mathbb{R}^I_+,\quad
\mathcal{A}(\overline{\bm{\omega}})\overline{\mathbf{v}}^{N}=0.
\end{equation*}
This contradicts the assumption that $\mathcal{A}(\overline{\bm{\omega}})$ is an $R_0$ tensor.
\end{proof}
\subsection{A sufficient condition for ${\rm ERM}(\mathcal{A}(\cdot),\mathbf{q}(\cdot))$}
\begin{theorem}\label{STCP:theorem1}
Assume that there exists an $\overline{\bm{\omega}}\in\Omega$ such that $f(\overline{\bm{\omega}})>0$ and $\mathcal{A}(\overline{\bm{\omega}})$ is an $R_0$ tensor, where $f(\cdot)$ is the continuous probability function in {\rm (\ref{STCP:equation4})}. Then, for any number $\gamma>0$, the level set $\mathbb{D}(\gamma)$ is bounded, that is, $\mathbb{D}^{C}(\gamma)$ is unbounded.
\end{theorem}
\begin{proof}
By the continuity of $f$ and Lemma \ref{STCP:lemma1}, there exists a closed sphere $\mathbb{B}(\overline{\bm{\omega}},\delta)$ with $\delta>0$ and a constant $f_0>0$ such that $\mathcal{A}(\bm{\omega})$ is an $R_0$ tensor and $f(\bm{\omega})\geq f_0$
for all $\bm{\omega}\in\overline{\mathbb{B}}:=\mathbb{B}(\overline{\bm{\omega}},\delta)\cap\Omega$. Let us consider a sequence $\{\mathbf{x}_k\}\subset\mathbb{R}^I$. Then, by the continuity of $\mathcal{A}(\cdot)$, $\mathbf{q}(\cdot)$ and $\Phi$, for each $k$, there exists an $\bm{\omega}_k\in\overline{\mathbb{B}}$ such that
$\|\Phi(\mathbf{x}_k,\bm{\omega}_k)\|_2=\min_{\bm{\omega}\in \overline{\mathbb{B}}}\|\Phi(\mathbf{x}_k,\bm{\omega})\|_2$.
It then follows that
$$
G(\mathbf{x}_k) \geq\int_{\overline{\mathbb{B}}}
\|\Phi(\mathbf{x}_k,\bm{\omega})\|_2^2f(\bm{\omega})d\bm{\omega}
\geq\|\Phi(\mathbf{x}_k,\bm{\omega}_k)\|_2^2f_0\int_{\overline{\mathbb{B}}}d\bm{\omega}
\geq f_0C\|\Phi(\mathbf{x}_k,\bm{\omega}_k)\|_2^2,
$$
where $C=\int_{\overline{\mathbb{B}}}d\bm{\omega}$. To prove the lemma, it suffices to show that $\|\Phi(\mathbf{x}_k,\bm{\omega}_k)\|_2\rightarrow+\infty$ whenever $\|\mathbf{x}_k\|_2\rightarrow+\infty$.

 Assume that $\|\mathbf{x}_k\|_2\rightarrow+\infty$, it is not difficult to see that $x_{k,i}\rightarrow-\infty$ or $(\mathcal{A}(\bm{\omega}_k)\mathbf{x}_k^{N-1}+\mathbf{q}(\bm{\omega}_k))_i\rightarrow-\infty$ for some $i$, then $|\phi((\mathcal{A}(\bm{\omega}_k)\mathbf{x}_k^{N-1}+\mathbf{q}(\bm{\omega}_k))_i,x_{k,i})|\rightarrow+\infty$ and hence $\|\Phi(\mathbf{x}_k,\bm{\omega}_k)\|_2\rightarrow+\infty$. So we only need to consider the case where both $\{x_{k,i}\}$ and $\{(\mathcal{A}(\bm{\omega}_k)\mathbf{x}_k^{N-1}+\mathbf{q}(\bm{\omega}_k))_i\}$ are bounded below for all $i$. Then, by dividing each element of these sequences by $\|\mathbf{x}_k\|_2$ and passing to the limit, we obtain
\begin{equation*}
(\mathcal{A}(\widehat{\bm{\omega}})\widehat{\mathbf{v}}^{N-1})_i\geq0,\quad \widehat{v}_i\geq0,\quad i=1,2,\dots,I,
\end{equation*}
where $\widehat{\bm{\omega}}$ and $\widehat{\mathbf{v}}$ are accumulation points of $\{\bm{\omega}_k\}$ and $\left\{\frac{\mathbf{x}_k}{\|\mathbf{x}_k\|_2}\right\}$, respectively.

Note that $\widehat{\mathbf{v}}\in \overline{\mathbb{B}}$ and $\|\widehat{\mathbf{v}}\|_2=1$. Since $\mathcal{A}(\widehat{\bm{\omega}})$ is an $R_0$ tensor, there must exist some $i$ such that $(\mathcal{A}(\widehat{\bm{\omega}})\widehat{\mathbf{v}}^{N-1})_i\geq0$ and $\widehat{v}_i\geq0$. This implies $(\mathcal{A}(\bm{\omega}_k)\mathbf{x}_k^{N-1}+\mathbf{q}(\bm{\omega}_k))_i\rightarrow+\infty$ and $x_{k,i}\rightarrow+\infty$, which in turn implies $|\phi((\mathcal{A}(\bm{\omega}_k)\mathbf{x}_k^{N-1}+\mathbf{q}(\bm{\omega}_k))_i,x_{k,i})|\rightarrow+\infty$. Hence we have $\|\Phi(\mathbf{x}_k,\bm{\omega}_k)\|_2\rightarrow+\infty$. This completes the proof.
\end{proof}
The following theorem shows that the converse of Lemma \ref{STCP:theorem1} is also true, when $\mathcal{A}(\bm{\omega})\equiv\mathcal{A}$ and $\mathbf{q}(\bm{\omega})$ is a linear function of $\bm{\omega}$.
\begin{theorem}
Suppose that $\mathcal{A}(\bm{\omega})\equiv\mathcal{A}$. If $\mathcal{A}$ is not an $R_0$ tensor and $\mathbf{q}(\bm{\omega})$ is a linear function of $\bm{\omega}$, then there is a $\gamma>0$ such that the level set $\mathbb{D}(\gamma)$ is unbounded.
\end{theorem}
\begin{proof}
Since $\mathcal{A}$ is not an $R_0$ tensor, there exists a nonzero $\mathbf{x}\in\mathbb{R}^I$ such that
\begin{equation*}
\mathbf{x}\in\mathbb{R}^{I}_+,\quad \mathcal{A}\mathbf{x}^{N-1}\in\mathbb{R}^{I}_+,\quad \mathcal{A}\mathbf{x}^{N}=0,
\end{equation*}
which in particular implies that either $x_i=0$ or $(\mathcal{A}\mathbf{x}^{N-1})_i=0$ holds for each $i$. Hence we have
\begin{equation*}
\min((\mathcal{A}\mathbf{x}^{N-1}+\mathbf{q}(\bm{\omega}))_i,x_i)=
\begin{cases}
\begin{array}{ll}
  0 & x_i=0,\ (\mathcal{A}\mathbf{x}^{N-1}+\mathbf{q}(\bm{\omega}))_i\geq0; \\
  (\mathcal{A}\mathbf{x}^{N-1}+\mathbf{q}(\bm{\omega}))_i & x_i=0,\ (\mathcal{A}\mathbf{x}^{N-1}+\mathbf{q}(\bm{\omega}))_i\leq0; \\
  x_i & (\mathcal{A}\mathbf{x}^{N-1})_i=0,\ (\mathbf{q}(\bm{\omega}))_i\geq x_i; \\
  (\mathbf{q}(\bm{\omega}))_i & (\mathcal{A}\mathbf{x}^{N-1})_i=0,\ (\mathbf{q}(\bm{\omega}))_i\leq x_i.
\end{array}
\end{cases}
\end{equation*}
  Since $\mathcal{A}\mathbf{x}^{N-1}\in\mathbb{R}^I_+$, we have $|(\mathcal{A}\mathbf{x}^{N-1}+\mathbf{q}(\bm{\omega}))_i|\leq|(\mathbf{q}(\bm{\omega}))_i|$ whenever $(\mathcal{A}\mathbf{x}^{N-1}+\mathbf{q}(\bm{\omega}))_i\leq0$. Thus it follows from (\ref{STCP:equation2}) that
\begin{equation*}
\frac{1}{\sqrt{2}+2}|(\Phi_2(\mathbf{x},\bm{\omega}))_i|\leq|(\Phi_1(\mathbf{x},\bm{\omega}))_i|
=|\min((\mathcal{A}\mathbf{x}^{N-1}+\mathbf{q}(\bm{\omega}))_i,x_i)|\leq|(\mathbf{q}(\bm{\omega}))_i|
\end{equation*}
and hence we have
$$G(\mathbf{x})\leq(\sqrt{2}+2)^2\int_{\Omega}\|\mathbf{q}(\bm{\omega})\|_2^2dF(\bm{\omega}):=\gamma.$$
Since by assumption $\mathbf{q}(\bm{\omega})$ is a linear function of $\bm{\omega}$, it follows from assumption on $f(\bm{\omega})$ that we have $\gamma<\infty$.

Since the argument above holds for $\lambda\mathbf{x}$ with any $\lambda>0$, that is, $G(\lambda\mathbf{x})\leq\gamma$, we complete the proof.
\end{proof}
\section{Stochastic $R_0$ tensors for STCPs}\label{STCP:section4}
In this section, we use $\Phi(\mathbf{x},\bm{\omega})$ to denote $\Phi_1(\mathbf{x},\bm{\omega})$. We first give several equivalent conditions for the stochastic $R_0$ tensors. We then study the solution set of ${\rm ERM}(\mathcal{A}(\cdot),\mathbf{q}(\cdot))$ by the means of stochastic $R_0$ tensors.
\subsection{Stochastic $R_0$ tensors}
$\mathbf{A}(\cdot)\in \mathbb{R}^{I\times I}$ is called a stochastic $R_0$ matrix \cite[Definition 2.1]{slcp_fang_SJO_2007}, if
\begin{equation*}
\mathbf{x}\in\mathbb{R}^I_+,\quad \mathbf{A}(\bm{\omega})\mathbf{x}\in\mathbb{R}^I_+,\quad
\mathbf{x}^\top\mathbf{A}(\bm{\omega})\mathbf{x}=0\quad \text{a.e.}\quad
\Rightarrow\quad \mathbf{x}=\mathbf{0}_I.
\end{equation*}
Meanwhile, a stochastic $R_0$ tensor is formally defined as follows.
\begin{definition}
$\mathcal{A}(\cdot)$ is called a stochastic $R_0$ tensor, if
\begin{equation*}
\mathbf{x}\in\mathbb{R}^I_+,\quad \mathcal{A}(\bm{\omega})\mathbf{x}^{N-1}\in\mathbb{R}^I_+,\quad
\mathcal{A}(\bm{\omega})\mathbf{x}^{N}=0,\quad \text{a.e.}\quad
\Rightarrow\quad \mathbf{x}=\mathbf{0}_I.
\end{equation*}
\end{definition}
If $\Omega$ only contains a single realization, the definition of a stochastic $R_0$ tensor reduces to that of an $R_0$ tensor. If $N=2$, the definition of a stochastic $R_0$ tensor reduces to that of a stochastic $R_0$ matrix.
\begin{theorem}\label{STCP:theorem2}
The following statements are equivalent.
\begin{enumerate}[{\rm (i)}]
\item $\mathcal{A}(\cdot)$ is a stochastic $R_0$ tensor.
\item For any nonzero $\mathbf{x}\in\mathbb{R}^I_+$, at least one of the following two conditions is satisfied:
    \begin{enumerate}[{\rm (a)}]
    \item $\mathbb{P}\{\bm{\omega}:(\mathcal{A}(\bm{\omega})\mathbf{x}^{N-1})_i\neq0\}>0$ for some $i\in\mathbb{J}(\mathbf{x})$;
    \item $\mathbb{P}\{\bm{\omega}:(\mathcal{A}(\bm{\omega})\mathbf{x}^{N-1})_i<0\}>0$ for some $i\in\mathbb{I}(\mathbf{x})$.
    \end{enumerate}
\item ${\rm ERM}(\mathcal{A}(\cdot),\mathbf{q}(\cdot))$ with $\mathbf{q}(\bm{\omega})\equiv\mathbf{0}_I$ has zero as its unique global solution.
\end{enumerate}
\end{theorem}
\begin{proof}
The proof is given in the order ${\rm (i)}\Rightarrow{\rm (iii)}\Rightarrow{\rm (ii)}\Rightarrow{\rm (i)}$.

${\rm (i)}\Rightarrow{\rm (iii)}$: It is easy to see that ${\bf 0}_I$ is a global solution of ${\rm ERM}(\mathcal{A}(\cdot),\mathbf{q}(\cdot))$ with $\mathbf{q}(\bm{\omega})\equiv\mathbf{0}_I$, since $G(\mathbf{x})\geq0$ for all $\mathbf{x}\in\mathbb{R}^I_+$ and $G(\mathbf{0}_I)=0$. Now we show that the uniqueness of the solution. Let $\widetilde{\mathbf{x}}\in\mathbb{R}^I_+$ be an arbitrary vector such that $G(\widetilde{\mathbf{x}})=0$. By the definition of $G$, we have
\begin{equation*}
\Phi(\widetilde{\mathbf{x}},\bm{\omega})=\min(\mathcal{A}(\bm{\omega})\widetilde{\mathbf{x}}^{N-1},
\widetilde{\mathbf{x}})=0,\quad \text{a.e.},
\end{equation*}
which implies
\begin{equation*}
\widetilde{\mathbf{x}}\in\mathbb{R}^I_+,\quad\mathcal{A}(\bm{\omega})
\widetilde{\mathbf{x}}^{N-1}\in\mathbb{R}^I_+,\quad \mathcal{A}(\bm{\omega})
\widetilde{\mathbf{x}}^{N}=0,\quad\text{a.e.}
\end{equation*}
By the definition of a stochastic $R_0$ tensor, we deduce $\widetilde{\mathbf{x}}=\mathbf{0}_I$.

${\rm (iii)}\Rightarrow{\rm (ii)}$: Suppose (ii) does not hold, that is, there exists a nonzero $\mathbf{x}_0\in\mathbb{R}^I_+$ such that
\begin{equation*}
\begin{split}
\mathbb{P}\{\bm{\omega}:(\mathcal{A}(\bm{\omega})\mathbf{x}_0^{N-1})_i=0\}&=1\text{ for all }i\in\mathbb{J}(\mathbf{x}_0),\\
\mathbb{P}\{\bm{\omega}:(\mathcal{A}(\bm{\omega})\mathbf{x}_0^{N-1})_i\geq0\}&=1\text{ for all }i\in\mathbb{I}(\mathbf{x}_0).
\end{split}
\end{equation*}
Then it follows from $\mathbf{q}(\bm{\omega})\equiv\mathbf{0}_I$ that $G(\mathbf{x}_0)=0$. Moreover, it is easy to see that for any $\lambda>0$, $\lambda\mathbf{x}_0$ is a solution of ${\rm ERM}(\mathcal{A}(\cdot),\mathbf{0}_I)$, i.e., $\mathbf{0}_I$ is not the unique solution of ${\rm ERM}(\mathcal{A}(\cdot),\mathbf{0}_I)$.

${\rm (ii)}\Rightarrow{\rm (i)}$: Assume that there exists a nonzero vector $\mathbf{x}\in\mathbb{R}^{I}$ such that
\begin{equation*}
\mathbf{x}\in\mathbb{R}^I_+,\quad \mathcal{A}(\bm{\omega})\mathbf{x}^{N-1}\in\mathbb{R}^I_+,\quad
\mathcal{A}(\bm{\omega})\mathbf{x}^{N}=0,\quad \text{a.e.}
\end{equation*}
Then, since $\mathcal{A}(\bm{\omega})\mathbf{x}^{N}=0$, we have for almost all $\bm{\omega}$, $(\mathcal{A}(\bm{\omega})\mathbf{x}^{N-1})_i=0$ for all $i\in\mathbb{J}(\mathbf{x})$ and $(\mathcal{A}(\bm{\omega})\mathbf{x}^{N-1})_i\geq0$ for all $i\in\mathbb{I}(\mathbf{x})$. This contradicts (ii).
\end{proof}
For $\nu>0$, let us denote $\mathbb{B}_{\Omega}(\overline{\bm{\omega}},\nu)=\{\bm{\omega}\in \Omega:\|\bm{\omega}-\overline{\bm{\omega}}\|_2<\nu\}$ and
\begin{equation*}
{\rm supp}\Omega:=\left\{\overline{\bm{\omega}}\in\Omega:
\int_{\mathbb{B}_{\Omega}(\overline{\bm{\omega}},\nu)\cap\Omega}
dF(\bm{\omega})>0\text{ for any }\nu>0\right\}.
\end{equation*}
Here ${\rm supp}\Omega$ is called the support set of $\Omega$.
\begin{corollary}
Suppose that $\mathcal{A}(\cdot)$ is a continuous function of $\bm{\omega}$. Then $\mathcal{A}(\cdot)$ is a stochastic $R_0$ tensor if and only if for any nonzero $\mathbf{x}\in\mathbb{R}^I_+$, at least one of the following two conditions is satisfied:
\begin{enumerate}[{\rm (a)}]
\item there exists $\overline{\bm{\omega}}\in{\rm supp}\Omega$ such that $(\mathcal{A}(\overline{\bm{\omega}})\mathbf{x}^{N-1})_i\neq0$ for some $i\in\mathbb{J}(\mathbf{x})$;
\item there exists $\overline{\bm{\omega}}\in{\rm supp}\Omega$ such that $(\mathcal{A}(\overline{\bm{\omega}})\mathbf{x}^{N-1})_i<0$ for some $i\in\mathbb{I}(\mathbf{x})$.
\end{enumerate}
\end{corollary}
\begin{proof}
By the continuity of $\mathcal{A}(\cdot)$ and the definition of ${\rm supp}\Omega$, conditions (a) and (b) in this corollary imply (a) and (b) in Theorem \ref{STCP:theorem2} (ii), respectively.
\end{proof}
\begin{corollary}
Suppose that $\mathcal{A}(\cdot)$ is a continuous function of $\bm{\omega}$ and $\mathcal{A}(\overline{\bm{\omega}})\in T_{N,I}$ is an $R_0$ tensor for some $\overline{\bm{\omega}}\in\Omega$. Then $\mathcal{A}(\cdot)$ is a stochastic $R_0$ tensor.
\end{corollary}

The following example shows that the condition that $\mathcal{A}(\cdot)$ is a stochastic $R_0$ tensor is weaker than the condition that $\mathcal{A}(\cdot)$ is a continuous function of $\bm{\omega}$ and $\mathcal{A}(\overline{\bm{\omega}})$ is an $R_0$ tensor for some $\overline{\bm{\omega}}\in\Omega$.
\begin{example}
Let
\begin{equation*}
\begin{split}
\mathcal{A}(\overline{\omega})(1,:,:)&=
\begin{pmatrix}
-2\omega&\omega-|\omega|&0\\
\omega-|\omega|&-2\omega&0\\
0&0&\omega+|\omega|
\end{pmatrix},\\
\mathcal{A}(\overline{\omega})(2,:,:)&=
\begin{pmatrix}
0&0&0\\
0&\omega+|\omega|&-2\omega\\
0&-2\omega&\omega+|\omega|\\
\end{pmatrix},\
\mathcal{A}(\overline{\omega})(3,:,:)=
\begin{pmatrix}
0&0&0\\
0&0&0\\
0&0&0\\
\end{pmatrix},
\end{split}
\end{equation*}
where $\omega\in\Omega=[-0.5,0.5]$ and $\omega$ is uniformly distributed on $\Omega$. Clearly, for $\omega<0$, we have
\begin{equation*}
\mathcal{A}(\overline{\omega})(1,:,:)=
\begin{pmatrix}
-2\omega&2\omega&0\\
2\omega&-2\omega&0\\
0&0&0
\end{pmatrix},\
\mathcal{A}(\overline{\omega})(2,:,:)=
\begin{pmatrix}
0&0&0\\
0&0&-2\omega\\
0&-2\omega&0\\
\end{pmatrix}.
\end{equation*}
Then $\mathbf{x}=(1,1,0)^\top$ satisfies $\mathcal{A}(\omega)\mathbf{x}^2={\bf0}$. On the other hand, for $\omega>0$, we have
\begin{equation*}
\mathcal{A}(\overline{\omega})(1,:,:)=
\begin{pmatrix}
-2\omega&0&0\\
0&-2\omega&0\\
0&0&2\omega
\end{pmatrix},\
\mathcal{A}(\overline{\omega})(2,:,:)=
\begin{pmatrix}
0&0&0\\
0&2\omega&-2\omega\\
0&-2\omega&2\omega\\
\end{pmatrix}.
\end{equation*}
Then $\mathbf{x}=(0,1,1)^\top$ satisfies $\mathcal{A}(\omega)\mathbf{x}^2={\bf0}$. In this example, there is no $\omega\in\Omega$ such that $\mathcal{A}(\omega)$ is an $R_0$ tensor. However, $\mathcal{A}(\cdot)$ is a stochastic $R_0$ tensor as verified by Theorem {\rm \ref{STCP:theorem2}} {\rm (ii)}.

For any nonzero $\mathbf{x}\in\mathbb{R}^3_+$, if $x_1\neq0$ or $x_2\neq0$, then for any $\omega>0$, $(\mathcal{A}(\omega)\mathbf{x}^2)_1=-2\omega (x_1^2+x_2^2)<0$. If $x_1\neq0$ but $x_2=0$, then for any $\omega<0$, $(\mathcal{A}(\omega)\mathbf{x}^2)_1=-2\omega x_1^2<0$. If only $x_2\neq x_3$, then for any $\omega>0$, $(\mathcal{A}(\omega)\mathbf{x}^2)_2=-2\omega (x_2-x_3)^2<0$.
\end{example}
The following proposition shows a relationship between $\mathcal{A}(\cdot)$ and $\overline{\mathcal{A}}:=\mathbb{E}\{\mathcal{A}(\omega)\}$.
\begin{proposition}\label{STCP:proposition1}
If $\overline{\mathcal{A}}\in T_{N,I}$ is an $R_0$ tensor, then $\mathcal{A}(\cdot)$ is a stochastic $R_0$ tensor.
\end{proposition}
\begin{proof}
If $\mathcal{A}(\cdot)$ were not a stochastic $R_0$ tensor, then by Theorem \ref{STCP:theorem2} (ii), there exists nonzero $\mathbf{x}_0\in\mathbb{R}^I_+$ such that for almost all $\bm{\omega}$, $(\mathcal{A}(\bm{\omega})\mathbf{x}_0^{N-1})_i=0$ for $i\in\mathbb{J}(\mathbf{x}_0)$ and $(\mathcal{A}(\bm{\omega})\mathbf{x}_0^{N-1})_i\geq0$ for $i\in\mathbb{I}(\mathbf{x}_0)$. Therefore, $(\overline{\mathcal{A}}\mathbf{x}_0^{N-1})_i=0$ for $i\in\mathbb{J}(\mathbf{x}_0)$ and $(\overline{\mathcal{A}}\mathbf{x}_0^{N-1})_i\geq0$ for $i\in\mathbb{I}(\mathbf{x}_0)$. This is impossible, since $\overline{\mathcal{A}}$ is an $R_0$ tensor.
\end{proof}
This proposition implies that for any given $\widetilde{\mathcal{A}}\in T_{N,I}$, if $\widetilde{\mathcal{A}}$ is an $R_0$ tensor, then $\mathcal{A}(\cdot)=\widetilde{\mathcal{A}}+\mathcal{A}_0(\cdot)$ with $\mathbb{E}\{\mathcal{A}_0(\bm{\omega})\}$ being the zero tensor is a stochastic $R_0$ tensor. The converse of this proposition is not true. The next proposition gives a way to construct a stochastic $R_0$ tensor $\mathcal{A}(\cdot)$ from a given $\widetilde{\mathcal{A}}$ which is not necessarily an $R_0$ tensor. Let
\begin{equation*}
\Xi(\mathcal{A}):=\left\{\mathbf{x}\in\mathbb{R}_+^I:\mathbf{x}\neq\mathbf{0}_I,
(\mathcal{A}\mathbf{x}^{N-1})_i=0,i\in\mathbb{J}(\mathbf{x}) \text{ and }(\mathcal{A}\mathbf{x}^{N-1})_i \geq 0,i\in\mathbb{I}(\mathbf{x})\right\}.
\end{equation*}
Obviously, if $\Xi(\widetilde{\mathcal{A}})=\emptyset$, then $\widetilde{\mathcal{A}}$ is an $R_0$ tensor, and hence, by Proposition \ref{STCP:proposition1}, $\mathcal{A}(\cdot)=\widetilde{\mathcal{A}}+\mathcal{A}_0(\cdot)$ with $\mathbb{E}\{\mathcal{A}_0(\bm{\omega})\}$ being the zero tensor is a stochastic $R_0$ tensor.

\begin{proposition}\label{STCP:proposition2}
Let $\widetilde{\mathcal{A}}$ and $\mathcal{A}_0(\cdot)$ be such that $\Xi(\widetilde{\mathcal{A}})\neq\emptyset$ and $\mathbb{E}\{\mathcal{A}_0(\bm{\omega})\}$ is the zero tensor. Suppose that for any $\mathbf{x}\in\Xi(\widetilde{\mathcal{A}})$, at least one of the following two conditions is satisfied:
\begin{enumerate}[{\rm (1)}]
\item For some $i\in\mathbb{J}(\mathbf{x})$, $\mathbb{E}\{((\mathcal{A}_0(\bm{\omega})\mathbf{x}^{N-1})_i)^2\}>0$;
\item For some $i\in\mathbb{I}(\mathbf{x})$, $\mathbb{P}\{\bm{\omega}:(\mathcal{A}_0(\bm{\omega})\mathbf{x}^{N-1})_i<-b\}>0$ for any $b>0$.
\end{enumerate}
Then $\mathcal{A}(\cdot)=\widetilde{\mathcal{A}}+\mathcal{A}_0(\cdot)$ is a stochastic $R_0$ tensor.
\end{proposition}
\begin{proof}
For $\mathbf{x}\in\Xi(\widetilde{\mathcal{A}})$, these two conditions imply that the conditions in Theorem \ref{STCP:theorem2} (ii) hold for $\mathcal{A}(\cdot)$. For $\mathbf{x}\notin\Xi(\widetilde{\mathcal{A}})$, the same conditions also hold trivially. So $\mathcal{A}(\cdot)=\widetilde{\mathcal{A}}+\mathcal{A}_0(\cdot)$ is a stochastic $R_0$ tensor.
\end{proof}

Proposition \ref{STCP:proposition2} suggests a way to obtain a stochastic $R_0$ tensor $\mathcal{A}(\cdot)$ from an arbitrary tensor $\widetilde{\mathcal{A}}\in T_{N,I}$. Specially, we can construct a simple stochastic tensor $\mathcal{A}_0(\cdot)$ such that $\widetilde{\mathcal{A}}+\mathcal{A}_0(\cdot)$ is a stochastic $R_0$ tensor, as illustrated in the following example.
\begin{example}
Consider $\mathcal{A}\in T_{3,5}$ with nonzero entries
\begin{equation*}
\begin{split}
a_{133}&=\ \ 1,\quad a_{144}=-2,\quad a_{155}=-3,\quad
a_{233}=\ \ 1,\quad a_{244}=-6,\quad a_{255}=-3,\\
a_{313}&=-1,\quad a_{323}=-1,\quad a_{414}=\ \ 2,\quad
a_{424}=\ \ 6,\quad a_{515}=\ \ 3,\quad a_{525}=\ \ 3.
\end{split}
\end{equation*}
Clearly, $\mathcal{A}$ is not an $R_0$ tensor and we have
\begin{equation*}
\begin{split}
&\Xi_1(\mathcal{A}):=\left\{\mathbf{x}\in\mathbb{R}_+^5:\mathbf{x}=(0,0,\lambda,\alpha,\beta)^\top,
\lambda>0,\alpha,\beta\geq0,\lambda^2-6\alpha^2-3\beta^2\geq0\right\};\\
&\Xi_2(\mathcal{A}):=\left\{\mathbf{x}\in\mathbb{R}_+^5:\mathbf{x}=(\alpha,\beta,0,0,0)^\top,
\alpha,\beta\geq0\right\}.
\end{split}
\end{equation*}
Note that for any $\mathbf{x}\in\Xi_2(\mathcal{A})$, $\mathcal{A}\mathbf{x}^2$ is the zero vector. Hence $\Xi_2(\mathcal{A})$ does not satisfies the assumption of Proposition \ref{STCP:proposition2}. Let $\omega_0\in\mathbb{R}$ be a random variable obeying the standard normal distribution. Suppose that the nonzero entries of $\mathcal{A}_0(\omega_0)\in T_{3,5}$ are
$$\mathcal{A}_0(\omega_0)(1,3,3)=0.5\omega_0,\quad
\mathcal{A}_0(\omega_0)(3,1,3)=-0.5\omega_0,\quad
\mathcal{A}_0(\omega_0)(3,3,1)=-0.5\omega_0.$$
Then for any $b>0$, $\mathbb{P}\{\omega_0:(\mathcal{A}_0(\omega_0)\mathbf{x}^2)_1<-b\}>0$ holds for any $\mathbf{x}\in\Xi_1(\mathcal{A})$. Hence by Proposition \ref{STCP:proposition2}, $\mathcal{A}+\mathcal{A}_0$ is a stochastic $R_0$ tensor.
\end{example}

The following proposition shows that the sum of a stochastic $R_0$ tensor $\mathcal{A}(\cdot)$ and a tensor $\mathcal{A}_1(\cdot)$ with $\mathbb{E}\{\mathcal{A}_1(\omega_1)\}$ being the zero tensor yields a stochastic $R_0$ tensor.
\begin{proposition}
Let $\bm{\omega}=(\omega_0,\omega_1)$ and $\widehat{\mathcal{A}}(\bm{\omega})=\mathcal{A}(\omega_0)+\mathcal{A}_1(\omega_1)$, where $\mathcal{A}(\cdot)$ is a stochastic $R_0$ tensor, $\mathbb{E}\{\mathcal{A}_1(\omega_1)\}$ is the zero tensor and $\mathcal{A}(\omega_0)$ is independent of $\mathcal{A}_1(\omega_1)$. Then $\widehat{\mathcal{A}}(\bm{\omega})$ is a stochastic $R_0$ tensor.
\end{proposition}
\begin{proof}
If $\widetilde{\mathcal{A}}:=\mathbb{E}\{\mathcal{A}(\omega_0)\}$ is an $R_0$ tensor, then from $\mathbb{E}\{\mathcal{A}_1(\omega_1)\}=\mathcal{O}$ and Proposition \ref{STCP:proposition1}, $\mathcal{A}(\omega_0)+\mathcal{A}_1(\omega_1)$ is a stochastic $R_0$ tensor. Otherwise, let $\mathcal{A}_0(\omega_0)=\mathcal{A}(\omega_0)-\widetilde{\mathcal{A}}$ and choose any $\mathbf{x}\in\Xi(\widetilde{\mathcal{A}})$. Suppose that the first condition of Proposition \ref{STCP:proposition2} holds for $\mathcal{A}_0(\omega_0)$. Since $\mathcal{A}(\omega_0)$ is independent of $\mathcal{A}_1(\omega_1)$, we have
\begin{equation*}
\mathbb{E}\{(((\mathcal{A}_0(\omega_0)+\mathcal{A}_1(\omega_1))\mathbf{x}^{N-1})_i)^2\}
=\mathbb{E}\{((\mathcal{A}_0(\omega_0)\mathbf{x}^{N-1})_i)^2\}
+\mathbb{E}\{((\mathcal{A}_1(\omega_1)\mathbf{x}^{N-1})_i)^2\}>0
\end{equation*}
for some $i\in\mathbb{J}(\mathbf{x})$. Now, suppose that the second condition of Proposition \ref{STCP:proposition2} holds for $\mathcal{A}_0(\omega_0)$, i.e., $\mathbb{P}\{\bm{\omega}:(\mathcal{A}_0(\omega_0)\mathbf{x}^{N-1})_i<-b\}>0$ for some $i\in\mathbb{I}(\mathbf{x})$. Note that
\begin{equation*}
\begin{split}
\mathbb{P}\{\bm{\omega}:&((\mathcal{A}_0(\omega_0)+\mathcal{A}_1(\omega_1))\mathbf{x}^{N-1})_i<-b\}\\
&\geq\mathbb{P}\{(\omega_0,\omega_1):(\mathcal{A}_0(\omega_0)\mathbf{x}^{N-1})_i<-b\text{ and }(\mathcal{A}_1(\omega_1)\mathbf{x}^{N-1})_i\leq0\}\\
&=\mathbb{P}\{\omega_0:(\mathcal{A}_0(\omega_0)\mathbf{x}^{N-1})_i<-b\}
\mathbb{P}\{\omega_1:(\mathcal{A}_1(\omega_1)\mathbf{x}^{N-1})_i\leq0\}.
\end{split}
\end{equation*}
Since $\mathbb{E}\{(\mathcal{A}_1(\omega_1)\mathbf{x}^{N-1})_i\}=0$, we have $\mathbb{P}\{\omega_1:(\mathcal{A}_1(\omega_1)\mathbf{x}^{N-1})_i\leq0\}>0$. Thus, we have $$\mathbb{P}\{\bm{\omega}:((\mathcal{A}_0(\omega_0)+\mathcal{A}_1(\omega_1))\mathbf{x}^{N-1})_i<-b\}>0,$$
i.e., the second condition of Proposition \ref{STCP:proposition2} holds for $\mathcal{A}_0(\omega_0)+\mathcal{A}_1(\omega_1)$. Since
\begin{equation*}
\widehat{\mathcal{A}}(\bm{\omega})=\mathcal{A}(\omega_0)+\mathcal{A}_1(\omega_1)
=\widetilde{\mathcal{A}}+\mathcal{A}_0(\omega)+\mathcal{A}_1(\omega_1),
\end{equation*}
Proposition \ref{STCP:proposition2} implies that $\widehat{\mathcal{A}}(\bm{\omega})$ is a stochastic $R_0$ tensor.
\end{proof}
\subsection{A sufficient and necessary condition for ${\rm ERM}(\mathcal{A}(\cdot),\mathbf{q}(\cdot))$}
\begin{theorem}\label{STCP:theorem3}
Let $\mathbf{q}(\cdot)$ be arbitrary. Then $G(\mathbf{x})\rightarrow\infty$ as $\|\mathbf{x}\|_2\rightarrow\infty$ with $\mathbf{x}\in\mathbb{R}^{I}_+$ if and only if $\mathcal{A}(\cdot)$ is a stochastic $R_0$ tensor.
\end{theorem}
\begin{proof}
First, we prove the ``if'' part. For a given $\mathbf{x}\in\mathbb{R}^I$, we denote $|\mathbf{x}|=(|x_1|,|x_2|,\dots,|x_I|)^\top$ and ${\rm sign}(\mathbf{x})=({\rm sign}(x_1),{\rm sign}(x_2)\dots,{\rm sign}(x_I))^\top$, where
\begin{equation*}
{\rm sign}(x_i)=\begin{cases}
\begin{split}
1,\quad& x_{i}>0,\\
0,\quad & x_i=0,\\
-1,\quad&x_i<0.
\end{split}
\end{cases}
\end{equation*}
Note that for any $a,b\in\mathbb{R}$, it follows from \cite[Theorem 3.1]{slcp_fang_SJO_2007} that
\begin{equation*}
\begin{split}
2\min(a,b)&=a+b-{\rm sign}(a-b)(a-b)\\
&=(1-{\rm sign}(a-b))a+(1+{\rm sign}(a-b))b,
\end{split}
\end{equation*}
and
\begin{equation*}
\begin{split}
4(\min(a,b))^2&=a(1-{\rm sign}(a-b))^2a+b(1+{\rm sign}(a-b))^2b+2b(1-{\rm sign}(a-b)^2)a\\
&=2a(1-{\rm sign}(a-b))a+2b(1+{\rm sign}(a-b))b.
\end{split}
\end{equation*}
For any $\mathbf{x}\in\mathbb{R}^I$ and $\bm{\omega}\in \Omega$, we define the diagonal matrix
\begin{equation*}
\mathbf{D}(\mathbf{x},\bm{\omega})={\rm diag}({\rm sign}(\mathcal{A}(\bm{\omega})\mathbf{x}^{N-1}+\mathbf{q}(\bm{\omega})-\mathbf{x})).\footnote{For a given $\mathbf{x}\in\mathbb{R}^I$, ${\rm diag}(\mathbf{x})\in\mathbb{R}^{I\times I}$ is a diagonal matrix such that its diagonal entries equal the entries of $\mathbf{x}$.}
\end{equation*}
Then we have
\begin{eqnarray}\label{STCP:equation3}
\|\Phi(\mathbf{x},\bm{\omega})\|_2^2&=&\frac{1}{2}(\mathcal{A}(\bm{\omega})\mathbf{x}^{N-1}+\mathbf{q}(\bm{\omega}))^\top
(\mathbf{I}_I-\mathbf{D}(\mathbf{x},\bm{\omega}))(\mathcal{A}(\bm{\omega})\mathbf{x}^{N-1}+\mathbf{q}(\bm{\omega}))\nonumber\\
&&+\frac{1}{2}\mathbf{x}^\top(\mathbf{I}_I+\mathbf{D}(\mathbf{x},\bm{\omega}))\mathbf{x}.
\end{eqnarray}

Consider an arbitrary $\mathbf{x}\in\mathbb{R}^I$ with $\|\mathbf{x}\|_2=1$. Suppose condition (a) in Theorem \ref{STCP:theorem2} (ii) holds. Choose $i\in\mathbb{J}(\mathbf{x})$ such that $\mathbb{P}\{\bm{\omega}:(\mathcal{A}(\bm{\omega})\mathbf{x}^{N-1})_i\neq0\}>0$. Then there exists a sufficiently large $K>0$ such that $\mathbb{P}\{\bm{\omega}:(\mathcal{A}(\bm{\omega})\mathbf{x}^{N-1})_i\neq0,|(\mathbf{q}(\bm{\omega}))_i|\leq K\}>0$.

First, we consider the case where $\mathbb{P}\{\bm{\omega}:(\mathcal{A}(\bm{\omega})\mathbf{x}^{N-1})_i<x_i,|(\mathbf{q}(\bm{\omega}))_i|\leq K\}>0$. Let
$$\Omega_1:=\{\bm{\omega}:(\mathcal{A}(\bm{\omega})\mathbf{x}^{N-1})_i<(1-\delta)x_i,|(\mathbf{q}(\bm{\omega}))_i|\leq K\},$$
where $\delta>0$. Then we have $\mathbb{P}\{\Omega_1\}>0$ whenever $\delta$ is sufficiently small. Moreover, for any sufficiently large $\lambda>0$, ${\rm sign}(\lambda^{N-1}(\mathcal{A}(\bm{\omega})\mathbf{x}^{N-1})_i+(\mathbf{q}(\bm{\omega}))_i-\lambda x_i)=-1$ for any $\bm{\omega}\in\Omega_1$. Therefore, by (\ref{STCP:equation4}) and (\ref{STCP:equation3}), we have
\begin{equation*}
G(\lambda\mathbf{x})>\int_{\Omega_1}(\lambda^{N-1}(\mathcal{A}(\bm{\omega})\mathbf{x}^{N-1})_i+(\mathbf{q}(\bm{\omega}))_i)^2dF(\bm{\omega})\rightarrow\infty\text{ as }\lambda\rightarrow\infty.
\end{equation*}

Next, we consider the case where $\mathbb{P}\{\bm{\omega}:(\mathcal{A}(\bm{\omega})\mathbf{x}^{N-1})_i>x_i,|(\mathbf{q}(\bm{\omega}))_i|\leq K\}>0$. Let
$$\Omega_2:=\{\bm{\omega}:(\mathcal{A}(\bm{\omega})\mathbf{x}^{N-1})_i<(1-\delta)x_i,|(\mathbf{q}(\bm{\omega}))_i|\leq K\}.$$
Then we have $\mathbb{P}\{\Omega_1\}>0$ whenever $\delta$ is sufficiently small. Moreover, for any sufficiently large $\lambda>0$, ${\rm sign}(\lambda^{N-1}(\mathcal{A}(\bm{\omega})\mathbf{x}^{N-1})_i+(\mathbf{q}(\bm{\omega}))_i-\lambda x_i)=1$ for any $\bm{\omega}\in\Omega_2$. Therefore, we have
\begin{equation*}
G(\lambda\mathbf{x})>\int_{\Omega_2}(\lambda x_i)^2dF(\bm{\omega})\rightarrow\infty\text{ as }\lambda\rightarrow\infty.
\end{equation*}

Finally, we consider the case where $\mathbb{P}\{\bm{\omega}:(\mathcal{A}(\bm{\omega})\mathbf{x}^{N-1})_i=x_i,|(\mathbf{q}(\bm{\omega}))_i|\leq K\}>0$. Let
$$\Omega_3:=\{\bm{\omega}:(\mathcal{A}(\bm{\omega})\mathbf{x}^{N-1})_i=x_i,|(\mathbf{q}(\bm{\omega}))_i|\leq K\}.$$
Then we have
\begin{equation*}
G(\lambda\mathbf{x})\geq\int_{\Omega_3}\{(\lambda x_i+(\mathbf{q}(\bm{\omega}))_i)^21_{\{(\mathbf{q}(\bm{\omega}))_i<0\}}+(\lambda x_i)^21_{\{(\mathbf{q}(\bm{\omega}))_i\geq0\}}\}dF(\bm{\omega})\rightarrow\infty\text{ as }\lambda\rightarrow\infty,
\end{equation*}
where for $(\mathbf{q}(\bm{\omega}))_i<0$, $1_{\{(\mathbf{q}(\bm{\omega}))_i<0\}}=1$ and for $(\mathbf{q}(\bm{\omega}))_i\geq0$, $1_{\{(\mathbf{q}(\bm{\omega}))_i\geq0\}}=1$. Hence we see that $G(\lambda\mathbf{x})\rightarrow\infty$ as $\lambda\rightarrow\infty$.

Now, we suppose condition (b) in Theorem \ref{STCP:theorem2} (ii) holds. Choose $i\in\mathbb{I}(\mathbf{x})$ such that $\mathbb{P}\{\bm{\omega}:(\mathcal{A}(\bm{\omega})\mathbf{x}^{N-1})_i<0\}>0$. Let
$$\Omega_4:=\{\bm{\omega}:(\mathcal{A}(\bm{\omega})\mathbf{x}^{N-1})_i<-\delta,|(\mathbf{q}(\bm{\omega}))_i|<K\}.$$
Then we have $\mathbb{P}\{\Omega_4\}>0$ for sufficiently small $\delta>0$ and sufficiently large $K>0$. Moreover, for any $\lambda>0$ large enough, $\lambda^{N-1}(\mathcal{A}(\bm{\omega})\mathbf{x}^{N-1})_i+(\mathbf{q}(\bm{\omega}))_i<0$ for $\bm{\omega}\in\Omega_4$. Thus we have
\begin{equation*}
\begin{split}
&(1-{\rm sign}((\lambda^{N-1}\mathcal{A}(\bm{\omega})\mathbf{x}^{N-1})_i+(\mathbf{q}(\bm{\omega}))_i))
((\lambda^{N-1}\mathcal{A}(\bm{\omega})\mathbf{x}^{N-1})_i+(\mathbf{q}(\bm{\omega}))_i)^2\\
&\quad=2((\lambda^{N-1}\mathcal{A}(\bm{\omega})\mathbf{x}^{N-1})_i+(\mathbf{q}(\bm{\omega}))_i)^2,
\end{split}
\end{equation*}
which yields
\begin{equation*}
G(\lambda\mathbf{x})>\int_{\Omega_4}(\lambda^{N-1}(\mathcal{A}(\bm{\omega})\mathbf{x}^{N-1})_i+(\mathbf{q}(\bm{\omega}))_i)^2dF(\bm{\omega})\rightarrow\infty\text{ as }\lambda\rightarrow\infty.
\end{equation*}

Since $\mathbf{x}$ is an arbitrary nonzero vector such that $\mathbf{x}\in\mathbb{R}^I_+$, we deduce from the above arguments that $G(\lambda\mathbf{x})\rightarrow\infty$ as $\|\mathbf{x}\|_2\rightarrow\infty$ with $\mathbf{x}\in\mathbb{R}^I_+$, provided that the statement (ii) in Theorem \ref{STCP:theorem2} holds.

Let us turn to proving the ``only if'' part. Suppose that $\mathcal{A}(\cdot)$ is not a stochastic $R_0$ tensor, i.e., there exists nonzero $\mathbf{x}_0\in\mathbb{R}^I_+$ such that  $(\mathcal{A}(\bm{\omega})\mathbf{x}_0^{N-1})_i=0$ for $i\in\mathbb{J}(\mathbf{x}_0)$ and $(\mathcal{A}(\bm{\omega})\mathbf{x}_0^{N-1})_i\geq0$ for $i\in\mathbb{I}(\mathbf{x}_0)$, a.e. For any $\lambda>0$, from (\ref{STCP:equation4}) and (\ref{STCP:equation3}), we have
\begin{eqnarray}\label{STCP:equation5}
G(\lambda\mathbf{x})&=&\frac{1}{2}\sum_{i=1}^I\mathbb{E}\left\{(1-{\rm sign}((\lambda^{N-1}\mathcal{A}(\bm{\omega})\mathbf{x}^{N-1})_i +(\mathbf{q}(\bm{\omega}))_i-\lambda x_i))
((\lambda^{N-1}\mathcal{A}(\bm{\omega})\mathbf{x}^{N-1})_i\right.\nonumber\\
&+&\left (\mathbf{q}(\bm{\omega}))_i-\lambda x_i)^2
+
(1-{\rm sign}((\lambda^{N-1}\mathcal{A}(\bm{\omega})\mathbf{x}^{N-1})_i
+(\mathbf{q}(\bm{\omega}))_i-\lambda x_i))(\lambda x_i)^2\right\}.
\end{eqnarray}
The $i$th term of the right-hand side of (\ref{STCP:equation5}) with $x_i\neq0$ equals to
\begin{eqnarray*}
&&\mathbb{E}\{(1-{\rm sign}((\mathbf{q}(\bm{\omega}))_i-\lambda x_i))(\mathbf{q}(\bm{\omega}))_i^2+(1-{\rm sign}((\mathbf{q}(\bm{\omega}))_i-\lambda x_i))(\lambda x_i)^2\}\\
&=&2\mathbb{E}\{(\mathbf{q}(\bm{\omega}))_i^21_{\{(\mathbf{q}(\bm{\omega}))_i\leq\lambda x_i\}}
+(\lambda x_i)^21_{\{(\mathbf{q}(\bm{\omega}))_i>\lambda x_i\}}\}\leq2\mathbb{E}\{(\mathbf{q}(\bm{\omega}))_i^2\},
\end{eqnarray*}
while the $i$th term of the right-hand side of (\ref{STCP:equation5}) with $x_i=0$ is
\begin{eqnarray*}
&&\mathbb{E}\{(1-{\rm sign}(\lambda^{N-1}(\mathcal{A}(\bm{\omega})\mathbf{x}^{N-1})_i+(\mathbf{q}(\bm{\omega}))_i))
(\lambda^{N-1}(\mathcal{A}(\bm{\omega})\mathbf{x}^{N-1})_i+(\mathbf{q}(\bm{\omega}))_i)^2\}\\
&=&2\mathbb{E}\{(\lambda^{N-1}(\mathcal{A}(\bm{\omega})\mathbf{x}^{N-1})_i+(\mathbf{q}(\bm{\omega}))_i)^2
1_{\{\lambda^{N-1}(\mathcal{A}(\bm{\omega})\mathbf{x}^{N-1})_i<-(\mathbf{q}(\bm{\omega}))_i\}}\}
\leq2\mathbb{E}\{(\mathbf{q}(\bm{\omega}))_i^2\},
\end{eqnarray*}
where the last inequality follows from $0>\lambda^{N-1}(\mathcal{A}(\bm{\omega})\mathbf{x}^{N-1})_i+(\mathbf{q}(\bm{\omega}))_i\geq(\mathbf{q}(\bm{\omega}))_i$, implying $$(\lambda^{N-1}(\mathcal{A}(\bm{\omega})\mathbf{x}^{N-1})_i+(\mathbf{q}(\bm{\omega}))_i)^2\geq(\mathbf{q}(\bm{\omega}))_i^2.$$ So, we have
$G(\lambda\mathbf{x})\leq2\mathbb{E}\{(\mathbf{q}(\bm{\omega}))_i^2\}$
for any $\lambda>0$.

Since $\mathbf{x}\in\mathbb{R}^I$ is nonzero, this particularly implies that $G$ is bounded above on a nonnegative ray in $\mathbb{R}^I$. This completes the proof of the ``only if'' part.
\end{proof}
The solution set of ${\rm ERM}(\mathcal{A}(\cdot),\mathbf{q}(\cdot))$ may be bounded, even if $\mathcal{A}(\cdot)$ is not a stochastic $R_0$ tensor. It depends on the distribution of $\mathbf{q}(\cdot)$, as shown in the following two propositions.
\begin{proposition}
If $\mathcal{A}(\cdot)$ is not a stochastic $R_0$ tensor, $\mathbb{P}\{\bm{\omega}:(\mathbf{q}({\bm \omega}))_i>0\}>0$ for some $i\in\mathbb{J}(\mathbf{x})$, and $\mathbb{P}\{\bm{\omega}:(\mathbf{q}({\bm \omega}))_i\geq0\}=1$ for some $i\in\mathbb{I}(\mathbf{x})$, where $\mathbf{x}\neq\mathbf{0}_I$ is any nonnegative vector at which the conditions {\rm (a)} and {\rm (b)} in Theorem {\rm \ref{STCP:theorem2}} {\rm (ii)} fail to hold, then the solution set of ${\rm ERM}(\mathcal{A}(\cdot),\mathbf{q}(\cdot))$ is bounded.
\end{proposition}
\begin{proof}
Note that
\begin{equation}\label{STCP:equation9}
G(\mathbf{0})=\mathbb{E}\{\|\Phi(\mathbf{0},\bm{\omega})\|_2^2\}=\sum_{i=1}^I\mathbb{E}\{(\mathbf{q}(\bm{\omega}))_i^2
1_{\{(\mathbf{q}(\bm{\omega}))_i<0\}}\}.
\end{equation}

For any nonzero $\mathbf{x}\in\mathbb{R}^I_+$ satisfying the conditions (a) and (b) in Theorem \ref{STCP:theorem2} (ii), the proof of Theorem \ref{STCP:theorem3} indicates that
\begin{equation}\label{STCP:equation12}
G(\lambda\mathbf{x})\rightarrow\infty \text{ as } \|\mathbf{x}\|_2\rightarrow\infty.
\end{equation}

Let $\mathbf{x}\in\mathbb{R}^I_+$ be nonzero which does not satisfy the conditions (a) and (b) in Theorem \ref{STCP:theorem2} (ii), i.e., $(\mathcal{A}(\bm{\omega})\mathbf{x}_0^{N-1})_i=0$ for $i\in\mathbb{J}(\mathbf{x}_0)$ and $(\mathcal{A}(\bm{\omega})\mathbf{x}_0^{N-1})_i\geq0$ for $i\in\mathbb{I}(\mathbf{x}_0)$, a.e. Then by (\ref{STCP:equation3}), we have
\begin{eqnarray}\label{STCP:equation6}
G(\lambda\mathbf{x})&=&\frac{1}{2}\sum_{i\in\mathbb{J}(\mathbf{x})}\mathbb{E}\left\{(1-{\rm sign}((\mathbf{q}(\bm{\omega}))_i-\lambda x_i))
((\mathbf{q}(\bm{\omega}))_i-\lambda x_i)^2+(1-{\rm sign}((\mathbf{q}(\bm{\omega}))_i-\lambda x_i))(\lambda x_i)^2\right\}\nonumber\\
&=&\sum_{i\in\mathbb{J}(\mathbf{x})}\{\mathbb{E}\{(\mathbf{q}(\bm{\omega}))_i^2\}
-\mathbb{E}\{1_{\{(\mathbf{q}(\bm{\omega}))_i-\lambda x_i>0\}}[(\mathbf{q}(\bm{\omega}))_i^2-(\lambda x_i)^2]\}\},
\end{eqnarray}
where the fist equality follows from the assumption that $\mathbb{P}\{\bm{\omega}:(\mathbf{q}(\bm{\omega}))_i\geq0\}=1$ for $i\in\mathbb{I}(\mathbf{x})$ and hence $(\mathcal{A}(\bm{\omega})\mathbf{x}_0^{N-1})_i+(\mathbf{q}(\bm{\omega}))_i\geq0$, a.e., for $i\in\mathbb{I}(\mathbf{x})$. Note that
\begin{equation*}
\begin{split}
0&\leq\mathbb{E}\{1_{\{(\mathbf{q}(\bm{\omega}))_i-\lambda x_i>0\}}[(\mathbf{q}(\bm{\omega}))_i^2-(\lambda x_i)^2]\}
=\mathbb{E}\{1_{\{(\mathbf{q}(\bm{\omega}))_i>\lambda x_i\}}[(\mathbf{q}(\bm{\omega}))_i^2-(\lambda x_i)^2]\}\\
&\leq\mathbb{E}\{1_{\{(\mathbf{q}(\bm{\omega}))_i>\lambda x_i\}}(\mathbf{q}(\bm{\omega}))_i^2\}\rightarrow\infty\text{ as }\lambda\rightarrow\infty,
\end{split}
\end{equation*}
which together with (\ref{STCP:equation6}) implies
\begin{equation}\label{STCP:equation7}
\lim_{\lambda\rightarrow\infty}G(\lambda\mathbf{x})=\sum_{i\in\mathbb{J}(\mathbf{x})}
\mathbb{E}\{(\mathbf{q}(\bm{\omega}))_i^2\}.
\end{equation}
On the other hand, for any nonzero $\mathbf{x}\in\mathbb{R}^{I}_+$, we have
\begin{equation}\label{STCP:equation8}
\sum_{i=1}^I\mathbb{E}\{(\mathbf{q}(\bm{\omega}))_i^21_{\{(\mathbf{q}(\bm{\omega}))_i<0\}}\}
=\sum_{i\in\mathbb{J}(\mathbf{x})}\mathbb{E}\{(\mathbf{q}(\bm{\omega}))_i^21_{\{(\mathbf{q}(\bm{\omega}))_i<0\}}\}
<\sum_{i\in\mathbb{J}(\mathbf{x})}
\mathbb{E}\{(\mathbf{q}(\bm{\omega}))_i^2\},
\end{equation}
where the equality follows from the assumption that $\mathbb{P}\{\bm{\omega}:(\mathbf{q}(\bm{\omega}))_i\geq0\}=1$ for $i\in\mathbb{I}(\mathbf{x})$ and the inequality follows from the assumption that $\mathbb{P}\{\bm{\omega}:(\mathbf{q}(\bm{\omega}))_i>0\}>0$ for $i\in\mathbb{J}(\mathbf{x})$. Combining (\ref{STCP:equation9}), (\ref{STCP:equation7}) and (\ref{STCP:equation8}), we have
\begin{equation}\label{STCP:equation13}
G(\mathbf{0})<\lim_{\lambda\rightarrow\infty}G(\lambda\mathbf{x}).
\end{equation}
Let $\Lambda=\{\mathbf{x}\in\mathbb{R}^I_+:G(\mathbf{x})\leq G(\mathbf{0})\}$. From (\ref{STCP:equation12}) and (\ref{STCP:equation13}), we have $\sup_{\mathbf{x}\in\Lambda}\|\mathbf{x}\|_2<+\infty$. Since any solution belongs to $\Lambda$, this implies that the solution set if bounded.
\end{proof}
\begin{proposition}
If $\mathcal{A}(\cdot)$ is not a stochastic $R_0$ tensor and, for any $i$, $\mathbb{P}\{\bm{\omega}:-b\leq (\mathbf{q}(\bm{\omega}))_i<0\}=1$ for some $b>0$, and $\mathbb{P}\{\bm{\omega}:(\mathbf{q}(\bm{\omega}))_i\neq0\text{ and }(\mathcal{A}(\bm{\omega})\widetilde{\mathbf{x}}^{N-1})_i=0\}=1$, where $\widetilde{\mathbf{x}}\neq\mathbf{0}_I$ is any nonnegative vector at which the conditions {\rm (a)} and {\rm (b)} in Theorem {\rm \ref{STCP:theorem2}} {\rm (ii)} fail to hold, then the solution set of ${\rm ERM}(\mathcal{A}(\cdot),\mathbf{q}(\cdot))$ is bounded.
\end{proposition}
\begin{proof}
Let $\widetilde{\mathbf{x}}\in\mathbb{R}^I_+$ be nonzero which does not satisfy the conditions (a) and (b) in Theorem \ref{STCP:theorem2} (ii). From (\ref{STCP:equation4}) and (\ref{STCP:equation3}), we have
\begin{eqnarray}\label{STCP:equation10}
G(\lambda\widetilde{\mathbf{x}})&=&\frac{1}{2}\sum_{i=1}^I\mathbb{E}\left\{(1-{\rm sign}((\lambda^{N-1}\mathcal{A}(\bm{\omega})\widetilde{\mathbf{x}}^{N-1})_i+(\mathbf{q}(\bm{\omega}))_i-\lambda \widetilde{x}_i))
((\lambda^{N-1}\mathcal{A}(\bm{\omega})\widetilde{\mathbf{x}}^{N-1})_i\right.\nonumber\\
&+&\left.(\mathbf{q}(\bm{\omega}))_i-\lambda \widetilde{x}_i)^2
+(1-{\rm sign}((\lambda^{N-1}\mathcal{A}(\bm{\omega})\widetilde{\mathbf{x}}^{N-1})_i+(\mathbf{q}(\bm{\omega}))_i-\lambda \widetilde{x}_i))(\lambda \widetilde{x}_i)^2\right\}.
\end{eqnarray}
For every $i\in\mathbb{J}(\widetilde{\mathbf{x}})$, we have $(\mathcal{A}(\bm{\omega})\widetilde{\mathbf{x}}^{N-1})_i=0$ and $q_i(\bm{\omega})=0$, a.e., and hence the $i$th term of the right-hand side of (\ref{STCP:equation10}) is zero for any $\lambda>0$. For every $i\in\mathbb{I}(\widetilde{\mathbf{x}})$, we have $(\mathcal{A}(\bm{\omega})\widetilde{\mathbf{x}}^{N-1})_i\geq0$ and $q_i(\bm{\omega})<0$, a.e., which implies
\begin{equation}\label{STCP:equation11}
\begin{split}
\mathbb{E}\{(1&-{\rm sign}(\lambda^{N-1}(\mathcal{A}(\bm{\omega})\widetilde{\mathbf{x}}^{N-1})_i+(\mathbf{q}(\bm{\omega}))_i))
(\lambda^{N-1}(\mathcal{A}(\bm{\omega})\widetilde{\mathbf{x}}^{N-1})_i+(\mathbf{q}(\bm{\omega}))_i)^2\}\\
=&2\mathbb{E}\{(\lambda(\mathcal{A}(\bm{\omega})\widetilde{\mathbf{x}}^{N-1})_i+(\mathbf{q}(\bm{\omega}))_i)^2
1_{\{\lambda^{N-1}(\mathcal{A}(\bm{\omega})\widetilde{\mathbf{x}}^{N-1})_i<-(\mathbf{q}(\bm{\omega}))_i,
(\mathcal{A}(\bm{\omega})\widetilde{\mathbf{x}}^{N-1})_i>0\}}\}\\
&\quad+2\mathbb{E}\{(\mathbf{q}(\bm{\omega}))_i^21_{(\mathcal{A}(\bm{\omega})\widetilde{\mathbf{x}}^{N-1})_i=0}\}.
\end{split}
\end{equation}
By assumption, the second term on the right-hand side of (\ref{STCP:equation11}) is zero for any $\lambda>0$, and
\begin{equation*}
\begin{split}
\mathbb{E}\{&(\lambda^{N-1}(\mathcal{A}(\bm{\omega})\widetilde{\mathbf{x}}^{N-1})_i+(\mathbf{q}(\bm{\omega}))_i)^2
1_{\{\lambda^{N-1}(\mathcal{A}(\bm{\omega})\widetilde{\mathbf{x}}^{N-1})_i<-(\mathbf{q}(\bm{\omega}))_i,
(\mathcal{A}(\bm{\omega})\widetilde{\mathbf{x}}^{N-1})_i>0\}}\}\\
&\leq b^2\mathbb{P}\{\bm{\omega}:0<\lambda^{N-1}(\mathcal{A}(\bm{\omega})\widetilde{\mathbf{x}}^{N-1})_i<b\}\rightarrow\infty
\text{ as }\lambda\rightarrow\infty.
\end{split}
\end{equation*}
Therefore, we obtain
\begin{equation*}
\lim_{\lambda\rightarrow\infty}G(\lambda\widetilde{\mathbf{x}})=0,
\end{equation*}
but for any $\mathbf{x}\in\mathbb{R}^I_+$, $G(\mathbf{x})\geq0$. So for any $\gamma>0$, the level set $\Lambda_\gamma:=\{\mathbf{x}\in\mathbb{R}^I_+:G(\mathbf{x})\leq\gamma\}$ is unbounded, which means the solution set is unbounded if it is not empty.
\end{proof}
From Theorem \ref{STCP:theorem3}, we have the following necessary and sufficient condition for the solution set of ${\rm ERM}(\mathcal{A}(\cdot),\mathbf{q}(\cdot))$ to be bounded for any $\mathbf{q}(\cdot)$.
\begin{theorem}
The solution set of ${\rm ERM}(\mathcal{A}(\cdot),\mathbf{q}(\cdot))$ is nonempty and bounded for any $\mathbf{q}(\cdot)$ if and only if $\mathcal{A}(\cdot)$ is a stochastic $R_0$ tensor.
\end{theorem}
\section{Concluding remarks}\label{STCP:section5}
In this paper, we introduced the definition of $R_0$ tensors and stochastic $R_0$ tensors to investigate the solution set of the expected residual minimization problem of the stochastic tensor complementarity problem. As we know, in order to study the TCP, many classes of the structured tensors are investigated in the recent years, such as $P$ tensors  \cite{tcp5}, $Q$ tensors  \cite{tcp8}, copositive tensors  \cite{Copositive}, and strictly (semi-)positive tensors \cite{tcp8}.  Gowda \cite{pcp} established Karamardian type results for the polynomial complementarity problem (PCP), as a special case of the NCP and a generalization of the TCP. In \cite{pcp}, Gowda also introduced the definition of degree of an $R_0$ tensor and showed that the degree of an $R_0$ tensor is one.

Similar to the STCP, the stochastic polynomial complementarity problem (SPCP) is to find a nonzero $\mathbf{x}\in\mathbb{R}^I_+$ such that
\begin{equation*}
\sum\limits_{k=1}^{K}\mathcal{A}_{k}(\bm{\omega})\mathbf{x}^{k-1}+\mathbf{q}(\bm{\omega})\in\mathbb{R}^I_+,\quad \sum\limits_{k=1}^{K}\mathcal{A}_{k}(\bm{\omega})\mathbf{x}^{k}+\mathbf{x}^\top\mathbf{q}(\bm{\omega})=0,
\end{equation*}
where $\mathcal{A}_k(\bm{\omega})\in T_{k,I}$ and $\mathbf{q}(\bm{\omega})\in\mathbb{R}^{I}$ for $\bm{\omega}\in\Omega$ with $k=1,2,\dots,K$. When the sample-path approach \cite{SNCP_introduction2} is applied to the SPCP, the approach is the same as the expected value method, which uses the expected function of the random function $\sum_{k=1}^{K}\mathcal{A}_{k}(\bm{\omega})\mathbf{x}^{k-1}+\mathbf{q}(\bm{\omega})$ and solves the deterministic problem
\begin{equation*}
\mathbf{x}\in\mathbb{R}^I_+,\quad \mathbb{E}\left\{\sum\limits_{k=1}^{K}\mathcal{A}_{k}(\bm{\omega})\mathbf{x}^{k-1}+\mathbf{q}(\bm{\omega})\right\}
\in\mathbb{R}^I_+,\quad \mathbb{E}\left\{\sum\limits_{k=1}^{K}\mathcal{A}_{k}(\bm{\omega})\mathbf{x}^{k}+\mathbf{x}^\top\mathbf{q}(\bm{\omega})\right\}=0.
\end{equation*}

Another generalization of the STCP is the stochastic extended vertical tensor complementarity problem (SEVTCP), which is to find a vector $\mathbf{x}\in\mathbb{R}^I$ such that
\begin{equation}\label{STCP:eqn1}
\min\{\Psi_0(\mathbf{x}),\Psi_1(\mathbf{x}),\dots,
\Psi_K(\mathbf{x})\}=\mathbf{0}_I,
\end{equation}
with $$
\Psi_k(\mathbf{x})=\mathbb{E}\{\mathcal{A}_k(\bm{\omega})\}\mathbf{x}^{N-1}
+\mathbb{E}\{\mathbf{q}_k(\bm{\omega})\},\quad k=0,1,\dots,K,$$
where $\mathcal{A}_k(\cdot):\Omega\rightarrow T_{N,I}$ and $\mathbf{q}_k(\cdot):\Omega\rightarrow\mathbb{R}^I$ are random mappings.

When $N=2$, the SEVTCP reduces to the  stochastic extended vertical linear complementarity problem \cite{sevlcp}, which is a natural extension of deterministic extended vertical linear complementarity problem \cite{golcp}. When $K=1$, $\mathbb{E}\{\mathcal{A}_0(\bm{\omega})\}=\mathcal{I}$ and $\mathbb{E}\{\mathbf{q}_0(\bm{\omega})\}={\bf 0}_I$, SEVTCP (\ref{STCP:eqn1}) reduces to the stochastic tensor complementarity problem ${\rm STCP}(\mathcal{A}_1(\cdot),\mathbf{q}_1(\cdot))$.

In the following, we present some open questions for the given STCP and SPCP, which will be studied in the future.
\begin{enumerate}[(a)]
\item Can we define the stochastic $P$ tensors, the stochastic copositive tensors and the stochastic strictly (semi-)positive tensors?
\item When the associated tensor is one of the above tensors, can we prove the solution set of the expected residual minimization problem of the STCP is nonempty and bounded?
\item Is there a relationship between a given SPCP and the associated STCP? Which condition can ensure the solution set of the expected residual minimization problem of the SPCP is nonempty and bounded?
\item Huang and Qi \cite{tcp3} reformulated a multilinear game (a class of $I$-person noncooperative games \cite{howson1972equilibria}) as a tensor complementarity problem and showed that finding a Nash equilibrium point of the multilinear game is equivalent to finding a solution of the resulted tensor complementarity problem. Which models can be reformulated as a stochastic tensor complementarity problem?
\end{enumerate}

\section*{Acknowledgements}
The authors would like to thank the editor and two referees for their detailed comments
and suggestions.

\end{document}